\documentclass[12pt]{article}

\usepackage{amsmath}
\usepackage{amsfonts}
\usepackage{color}
\usepackage{float}   
\usepackage{amssymb}
\usepackage{bm}    
\usepackage{graphicx}   
\usepackage{enumerate}
\usepackage{amsthm,amscd}
\usepackage[all]{xy}
\usepackage{multirow,makecell}
\usepackage{pifont}        
\usepackage{arydshln}    
\usepackage{booktabs}  

\usepackage{tikz,tcolorbox}

\makeatletter
\def\hlinew#1{%
\noalign{\ifnum0=`}\fi\hrule \@height #1 \futurelet
\reserved@a\@xhline}
\makeatother       

\allowdisplaybreaks

\usepackage{hyperref}

\newtheorem{theorem}{Theorem}[section]

\newtheorem{conjecture}[theorem]{Conjecture}

\newtheorem{problem}[theorem]{Problem}
\newtheorem{lemma}[theorem]{Lemma}

\begin{document}

\title{Spectral extremal graphs for the bowtie\thanks{
This paper was firstly announced in December, 2022,
and was later published on Discrete Mathematics 346 (2023) 113680.  
See \url{https://doi.org/10.1016/j.disc.2023.113680}. 
The research was supported by National Natural Science Foundation of China grant 11931002 and 12001544.
E-mail addresses: \url{505694365@qq.com} (Y\v{o}ngt\={a}o L\v{i}),
\url{lulugdmath@163.com} (L\'{u} L\v{u}),
\url{ypeng1@hnu.edu.cn} (Yu\`{e}ji\`{a}n P\'{e}ng, corresponding author).}  }

\author{Yongtao Li$^{\dag}$, Lu Lu$^{\dag}$, Yuejian Peng$^{*}$ \\[2ex]
{\small $^{\dag}$School of Mathematics and Statistics, Central South University} \\
{\small Changsha, Hunan, 410083, P.R. China} \\
{\small $^{*}$School of Mathematics, Hunan University} \\
{\small Changsha, Hunan, 410082, P.R. China }  
}

\maketitle

\begin{abstract} 
Let $F_k$ be the (friendship) graph obtained from $k$ triangles by sharing a common vertex. 
The $F_k$-free graphs of order $n$ which attain the maximal spectral radius was firstly characterized by Cioab\u{a}, Feng, Tait and Zhang [Electron. J. Combin. 27 (4) (2020)],
and later uniquely determined by Zhai, Liu and Xue [Electron. J. Combin. 29 (3) (2022)] under the condition that $n$ is sufficiently large. 
In this paper, we get rid of the condition on $n$ being sufficiently large if $k=2$. 
The graph $F_2$ is also known as the bowtie.  
We show that the unique $n$-vertex
$F_2$-free spectral extremal graph is  the balanced complete bipartite graph adding an edge in the vertex part with  smaller size if $n\ge 7$, and the condition $n\ge 7$ is tight. 
Our result is a spectral generalization of a theorem of  Erd\H{o}s, F\"{u}redi, Gould and Gunderson 
[J. Combin. Theory Ser. B 64 (1995)], which states that $\mathrm{ex}(n,F_2)=\left\lfloor {n^2}/{4} \right\rfloor +1$. 
Moreover, we study the spectral extremal problem for 
$F_k$-free graphs with given number of edges. 
In particular, we show that the
unique $m$-edge $F_2$-free spectral extremal graph is the join of $K_2$
with an independent set of $\frac{m-1}{2}$ vertices if $m\ge 8$, and the condition $m\ge 8$ is tight. 
\end{abstract}





\section{Introduction}

In this paper,
we shall use the following standard notation; see, e.g., the monograph \cite{BM2008}.
We consider only simple and undirected graphs. Let $G$ be a simple
graph with vertex set $V(G)=\{v_1, \ldots, v_n\}$ and edge set $E(G)=\{e_1, \ldots, e_m\}$.
We usually write $n$ and $m$ for the number of vertices and edges
respectively.
Let $N(v)$ or $N_G(v)$ be the set of neighbors of $v$,
and $d(v)$ or $d_G(v)$ be the degree of a vertex $v$ in $G$.
For a subset $A\subseteq V(G)$, we write $e(A)$ for the
number of edges with two endpoints in $A$,
and $N(A)$ for the union of neighborhoods of vertices of $A$.
For two disjoint sets $A,B$, we write $e(A,B)$ for the number of edges between $A$ and $B$.
Let $K_n$ be the complete graph on $n$ vertices, and
$K_{s,t}$ be the complete bipartite graph with parts of sizes
$s$ and $t$.
We write  $C_n$ and $P_n$ for the cycle and
path on $n$ vertices respectively.
Let $G \vee H$ be the join graph consisting of $G$ and $H$ in which each vertex of $G$
is adjacent to each vertex of $H$.
To illustrate the background and history,
we shall partition the introduction into the following three subsections.

\subsection{The classical extremal graph problems}

A graph $G$ is called $F$-free if it does not contain
an isomorphic copy of $F$ as a subgraph.
Apparently, every bipartite graph is $C_{2k+1}$-free
for every integer $k\ge 1$.
The  Mantel theorem \cite{Man1907} asserts that
the balanced complete bipartite graph  $K_{\lfloor \frac{n}{2}\rfloor, \lceil \frac{n}{2} \rceil }$
attains the maximum number of edges among all $n$-vertex triangle-free graphs.

\begin{theorem}[Mantel \cite{Man1907}, 1907] \label{thmMan}
If $G$ is a triangle-free graph on $n$ vertices, then
\begin{equation} \label{eq-man}
e(G) \le    \lfloor {n^2}/{4} \rfloor ,
\end{equation}
equality holds if and only if  $G$ is the balanced complete bipartite graph $K_{\lfloor \frac{n}{2}\rfloor, \lceil \frac{n}{2} \rceil }$.
\end{theorem}

Mantel's theorem has many interesting applications and generalizations in the literature; see, e.g.,
\cite[pp. 294--301]{Bollobas78} for standard proofs.
The {\em Tur\'{a}n number} of a graph $F$ is the maximum number of edges  in an $n$-vertex $F$-free graph, and
it is usually  denoted by $\mathrm{ex}(n, F)$. That is,
\[ \mathrm{ex} (n,F):=\max \bigl\{ e(G): |G|=n~\text{and}~F\nsubseteq G \bigr\}. \]
A graph on $n$ vertices with no subgraph $F$ and with $\mathrm{ex}(n, F)$ edges is called an {\em extremal graph} for $F$.
We write  $\mathrm{Ex}(n,F)$ for the set of all $n$-vertex $F$-free graphs
with maximum number of edges.
A well-known  theorem of Tur\'{a}n (see \cite{Bollobas78})
extended Theorem \ref{thmMan} and
proved  that  for every integer $r\ge 2$,
\begin{equation} \label{eq-Turan}
\mathrm{Ex}(n,K_{r+1})= \{T_r(n)\},
\end{equation}
where $T_r(n)$ is an $n$-vertex complete $r$-partite graph
whose parts have sizes as equal as possible.
In particular, we have $T_2(n)=K_{\lfloor \frac{n}{2}\rfloor, \lceil \frac{n}{2} \rceil }$.
The graph $T_r(n)$ is now known as Tur\'{a}n's graph.
Inspired by Turan's result (\ref{eq-Turan}),
the extremal graph theory was widely studied and achieved a rapid development in the past 100 years; see, e.g., \cite{FS13,Sim13} for recent comprehensive surveys.

Let $F_k$ denote the $k$-fan graph, which is the graph
consisting of $k$ triangles
that intersect in exactly one common vertex.
This graph is known as the friendship graph
because it is the only extremal graph in the
well-known Friendship Theorem \cite[Chapter 43]{AZ2014}.
Since the chromatic number $\chi (F_k)=3$,
a celebrated result due to  Erd\H{o}s, Stone and Simonovits (see \cite[p. 339]{Bollobas78})
gives an asymptotic value of the Tur\'{a}n number of $F_k$, which states that
\[ \mathrm{ex}(n,F_k)= n^2/4 + o(n^2). \]
A natural question is to determine the error term
$o(n^2)$ accurately.
In 1995, Erd\H{o}s, F\"{u}redi,
Gould and Gunderson \cite{Erdos95} proved  the following
result.

\begin{theorem}[Erd\H{o}s--F\"{u}redi--Gould--Gunderson \cite{Erdos95}, 1995]
\label{thmErdos95}
For every $k \geq 1$ and $n\geq 50k^2$,
\[ \mathrm{ex}(n, F_k)= \left\lfloor \frac {n^2}{4}\right \rfloor+ \left\{
\begin{array}{ll}
k^2-k, \quad~~  \mbox{if $k$ is odd,} \\
k^2-\frac32 k, \quad \mbox{if $k$ is even}.
\end{array}
\right. \]
\end{theorem}

The extremal graphs in $\mathrm{Ex}(n,F_k)$ are also determined in \cite{Erdos95} as follows.
\begin{itemize}
\item[(a)] For odd $k$,
the extremal graphs are constructed from $K_{\lfloor \frac{n}{2}\rfloor, \lceil \frac{n}{2} \rceil }$
by embedding two vertex-disjoint copies of the complete graph $K_k$ in one side.

\item[(b)]
For even $k$, the extremal graphs
are obtained by taking $K_{\lfloor \frac{n}{2}\rfloor, \lceil \frac{n}{2} \rceil }$ and embedding
a graph with $2k-1$ vertices, $k^2-\frac{3}{2} k$ edges and maximum degree $k-1$ in one side.
\end{itemize}

We remark here that for even $k$, the embedded subgraph is a nearly $(k-1)$-regular graph
of order $2k-1$ with degree
sequence $(k-1,\ldots ,k-1,k-2)$. Such a graph does exist for every even $k\ge 2$.
Moreover, for odd $k$ and even $n$,
the extremal graph is unique, it is obtained from
$K_{\frac{n}{2},\frac{n}{2}}$ by embedding two vertex-disjoint copies of $K_k$ in one part.
However, for other cases,
the extremal graph for $F_k$ is clearly not unique.
Additionally, it was conjectured by Erd\H{o}s, F\"{u}redi,
Gould and Gunderson
in the second paragraph after \cite[Theorem 2.1]{Erdos95} that
Theorem \ref{thmErdos95} gives $\mathrm{ex}(n,F_k)$
for all $n\ge 4k$, rather than $n\ge 50k^2$.
This conjecture remains open.

\subsection{The spectral extremal graph problems}

Let $G$ be a graph on $n$ vertices with $m$ edges.
Let $A(G)$ be the adjacency matrix of $G$.
Let $\lambda (G)$ be the spectral radius of $G$,
which is defined as the maximum of
modulus of eigenvalues of the adjacency matrix $A(G)$.
The classical extremal graph problems
usually study the maximum or minimum
number of edges that the extremal graphs can have.
Correspondingly,
we can study the spectral extremal  problem.
We denote by $\mathrm{ex}_{\lambda}(n,F)$
the largest eigenvalue of the adjacency matrix
in an $n$-vertex  graph that contains no copy of $F$.

In 2007, Nikiforov  \cite{Niki2007laa2} published the spectral Tur\'{a}n theorem, which asserts that
the $r$-partite Tur\'{a}n graph $T_r(n)$ is the unique $K_{r+1}$-free graph
attaining the maximum spectral radius.
In particular, the important triangle case provided a spectral version of  Theorem \ref{thmMan}.

\begin{theorem}[Nikiforov \cite{Niki2007laa2}, 2007] \label{thm-niki}
Let $G$ be a  triangle-free graph on $n$ vertices. Then
\begin{equation} \label{eq2}
\lambda (G)\le \lambda (K_{\lfloor \frac{n}{2}\rfloor, \lceil \frac{n}{2} \rceil } ),
\end{equation}
equality holds if and only if $G $
is a balanced complete bipartite graph
$K_{\lfloor \frac{n}{2}\rfloor, \lceil \frac{n}{2} \rceil }$.
\end{theorem}

Comparing the spectral result with the classical
result in terms of edges,
one can see that Theorem \ref{thm-niki}
is stronger than Theorem \ref{thmMan},
since
(\ref{eq2}) can deduce (\ref{eq-man}). Indeed, using (\ref{eq2}) and
Rayleigh's formula, we get  ${2e(G)}/{n} \le \lambda (G) \le \sqrt{\lfloor n^2/4\rfloor}$ and then $e(G) \le \bigl\lfloor \frac{n}{2}\sqrt{\lfloor n^2/4\rfloor} \bigr\rfloor = \lfloor n^2/4\rfloor$.
Such a deduction enriches the study of spectral extremal graph theory.
In the past few years,
Theorem \ref{thm-niki} stimulated the developments of the spectral extremal graph theory.
It is natural to consider  the spectral extremal  problems
for graphs with given number of vertices.
In view of this perspective,
various extensions and generalizations on
inequality (\ref{eq2})  have been obtained in the literature; see, e.g.,
\cite{Wil1986,Niki2007laa2,KN2014} for extensions
on $K_{r+1}$-free graphs;
see \cite{BN2007jctb} for relations
between cliques and spectral radius,
\cite{TT2017,LN2021outplanar} for outerplanar and planar graphs,
\cite{Tait2019} for the Colin de Verdi\`{e}re parameter,
\cite{LNW2021,LP2022second} for non-bipartite triangle-free graphs,
\cite{LG2021,LSY2022} for non-bipartite graphs without short odd cycles,
and \cite{NikifSurvey} for a comprehensive survey.

\medskip

Recall that $F_k$ is the graph obtained from $k$ triangles
by intersecting a common vertex,
and $\mathrm{Ex}(n,F_k)$ is the set of $n$-vertex $F_k$-free graphs
with the maximum number of edges.
For fixed $k\ge 2$ and sufficiently large order $n$,
the spectral extremal problem
for $F_k$ was recently
characterized  by Cioab\u{a}, Feng,
Tait and Zhang \cite{CFTZ20}.

\begin{theorem}[Cioab\u{a}--Feng--Tait--Zhang \cite{CFTZ20}, 2020]    \label{thmCFTZ20}
Let $k\ge 2$ and $G$ be an $F_k$-free graph on $n$ vertices.
For  sufficiently large $n$, if $G$ has the maximal spectral radius, then
\[  G \in  \mathrm{Ex}(n,F_k). \]
\end{theorem}

In 2022, Zhai, Liu and Xue \cite{ZLX2022} provided a {\it further}  characterization
of $G$ and
determined  the {\it unique} spectral extremal graph,
which is obtained from $K_{\lfloor \frac{n}{2}\rfloor, \lceil \frac{n}{2} \rceil }$ by embedding a graph $H$ in the part of smaller size
$\lfloor \frac{n}{2}\rfloor$, where $H=K_k \cup K_k$ for odd $k$, and $H$ is described as below for even $k$.
In this case, $H$ is a graph with vertex set $V(H)=\{w_0\} \cup A \cup B$
such that $N(w_0)=A$ and $|B|=|A| +2=k$.
Then we partition $A$ into $A_1\cup A_2$,
and $B$ into $\{u_0\}\cup B_1\cup B_2$
such that $|A_1|=|A_2|=|B_2|=\frac{k-2}{2}$ and $|B_1|=\frac{k}{2}$.
Finally, we join $k-1$ edges from $u_0$ to $A_1\cup B_1$,
$\frac{k-2}{2}$ independent edges between $B_2$ and $A_2$,
and some additional edges such that both $A$ and $B_1\cup B_2$
are cliques.
In fact, it is challenging and difficult to
characterize  the embedded small graph $H$ in the even case.
We refer the readers to \cite{ZLX2022} for more details, and
\cite{Chen03,LP2021,DKLNTW2021} for related problems,
and \cite{WKX2023} for a stronger result.

\begin{theorem}[Wang--Kang--Xue \cite{WKX2023}, 2023] \label{thm-WKX}
Let $r\ge 2$ be an integer, and $F$ be a graph with $\mathrm{ex}(n,F)= e(T_r(n)) + O(1)$. For sufficiently large $n$, if $G$ has the maximal spectral radius
over all $n$-vertex $F$-free graphs, then
\[  G\in \mathrm{Ex}(n,F). \]
\end{theorem}

\medskip
Note that the above results
(Theorems \ref{thmCFTZ20} and \ref{thm-WKX}) hold under the condition that $n$ is sufficiently large.
Nevertheless, the lower bound on $n$ is still unknown.
For example, given a specific integer $n$, say $n=10^8$,
we still do not know whether the conclusion is true.
In extremal combinatorics,
it is also an important work to determine the extremal configurations for all values of $n$,
instead of $n$ being large enough. For instance,
the following are three such works.

\begin{itemize}

\item[$\heartsuit$]
In 2005, F\"{u}redi and Simonovits \cite{FS2005},
Keevash and Sudakov \cite{KS2005}  independently
provided the Tur\'{a}n number of Fano plane for sufficiently large $n$.
In 2019,  Bellmann and Reiher \cite{BR2019} provided a complete solution
to the Tur\'{a}n number of Fano plane for every integer $n\ge 7$.
This confirmed a conjectura of S\'{o}s in a strong sense.

\item[$\heartsuit$]
Moreover, Tait and Tobin \cite{TT2017} proved that for sufficiently large $n$,
the outerplanar graph on $n$ vertices of
maximum spectral radius is $K_1\vee P_{n-1}$,
the join of a vertex and a path on $n-1$ vertices.
In 2021, Lin and Ning \cite{LN2021outplanar} proved further that
the same result still holds for all $n\ge 2$ except for $n=6$.
This gives a complete solution to the Cvetkovi\'{c}--Rowlinson conjecture.

\item[$\heartsuit$] 
In addition, Nikiforov \cite{Niki2009ejc}
proved that for every color-critical graph $F$ with 
$\chi (F)=r+1$, there exists  $n_0(F)$
such that if $n\ge n_0(F)$,
then the Tur\'{a}n graph $T_r(n)$
is the unique graph attaining the maximum spectral radius
among all $n$-vertex $F$-free graphs.
The bound $n_0(F)$ given in \cite{Niki2009ejc} is an exponential function on the order of $F$.
In 2022, Zhai and Lin \cite{ZL2022jgt} provided that
linear functions on $n_0(F)$ are valid for two  classes of graphs, namely,
books and theta graphs.
\end{itemize}

To some extent, it is meaningful to find an appropriate or smallest bound $n_0(k)$
such that for every $n\ge n_0(k)$, the spectral extremal graphs
are contained in  $ \mathrm{Ex}(n,F_k)$.

\begin{figure}[H]
\centering
\includegraphics[scale=0.12]{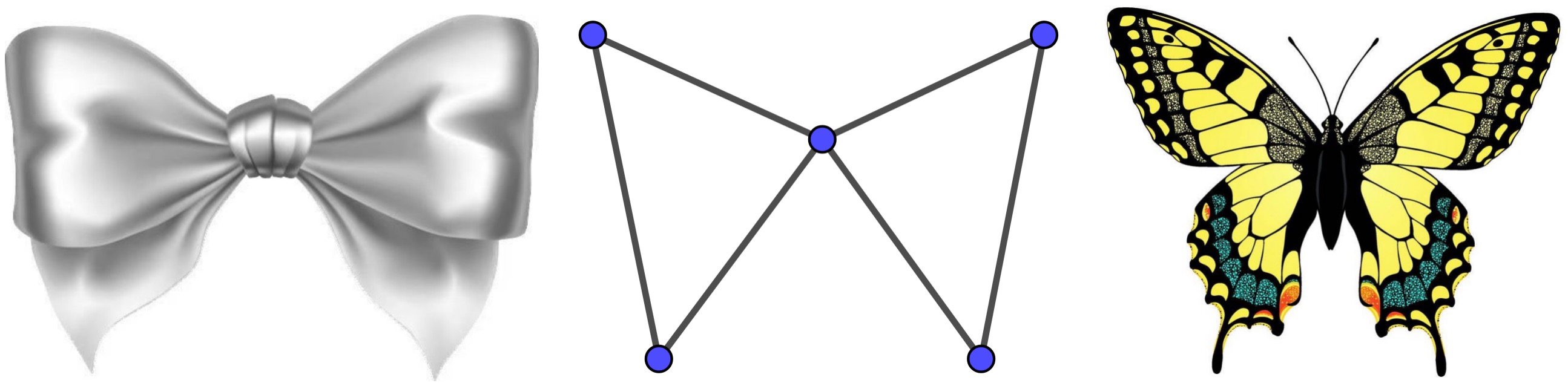} 
\caption{The bowtie graph $F_2$.}
\label{Fig-bowtie}
\end{figure}

In this paper, we investigate  such problems for
the bowtie $F_2$, which consists of two copies of $K_3$
merged at a vertex. We refer to this vertex as the
central vertex of the bowtie.
Despite the simple nature of the bowtie, bowtie-free graphs have been crucial in several areas of graph theory,
such as in Ramsey problem \cite{HN2018},
supersaturation problem \cite{KMP2020}
and chromatic number \cite{CCCFL2021}, etc.
In fact, the Tur\'{a}n number of the bowtie
was also considered by
Erd\H{o}s, F\"{u}redi, Gould and Gunderson \cite{Erdos95},
in which they presented the  exact value of $\mathrm{ex}(n,F_2)$ for every integer $n\ge 5$.

\begin{theorem}[Erd\H{o}s et al. \cite{Erdos95}, 1995] \label{EFGG-F2}
For $n\ge 5$,  one has
\[  \mathrm{ex}(n,F_2) = \left\lfloor {n^2}/{4} \right\rfloor +1. \]
\end{theorem}

Motivated by these results above, we will 
determine
the unique $F_2$-free extremal graph attaining the maximal spectral radius 
for all integers $n\ge 7$.
In what follows, we will establish a spectral extension on Theorem \ref{EFGG-F2}.
Denote by $K_{ \lfloor \frac{n}{2} \rfloor, \lceil \frac{n}{2} \rceil}^+$ the graph obtained from
the balanced complete bipartite graph $K_{ \lfloor \frac{n}{2} \rfloor, \lceil \frac{n}{2} \rceil}$
by embedding an edge in the smaller part;
see Figure \ref{fig-2}.

\begin{figure}[H]
\centering
\includegraphics[scale=0.85]{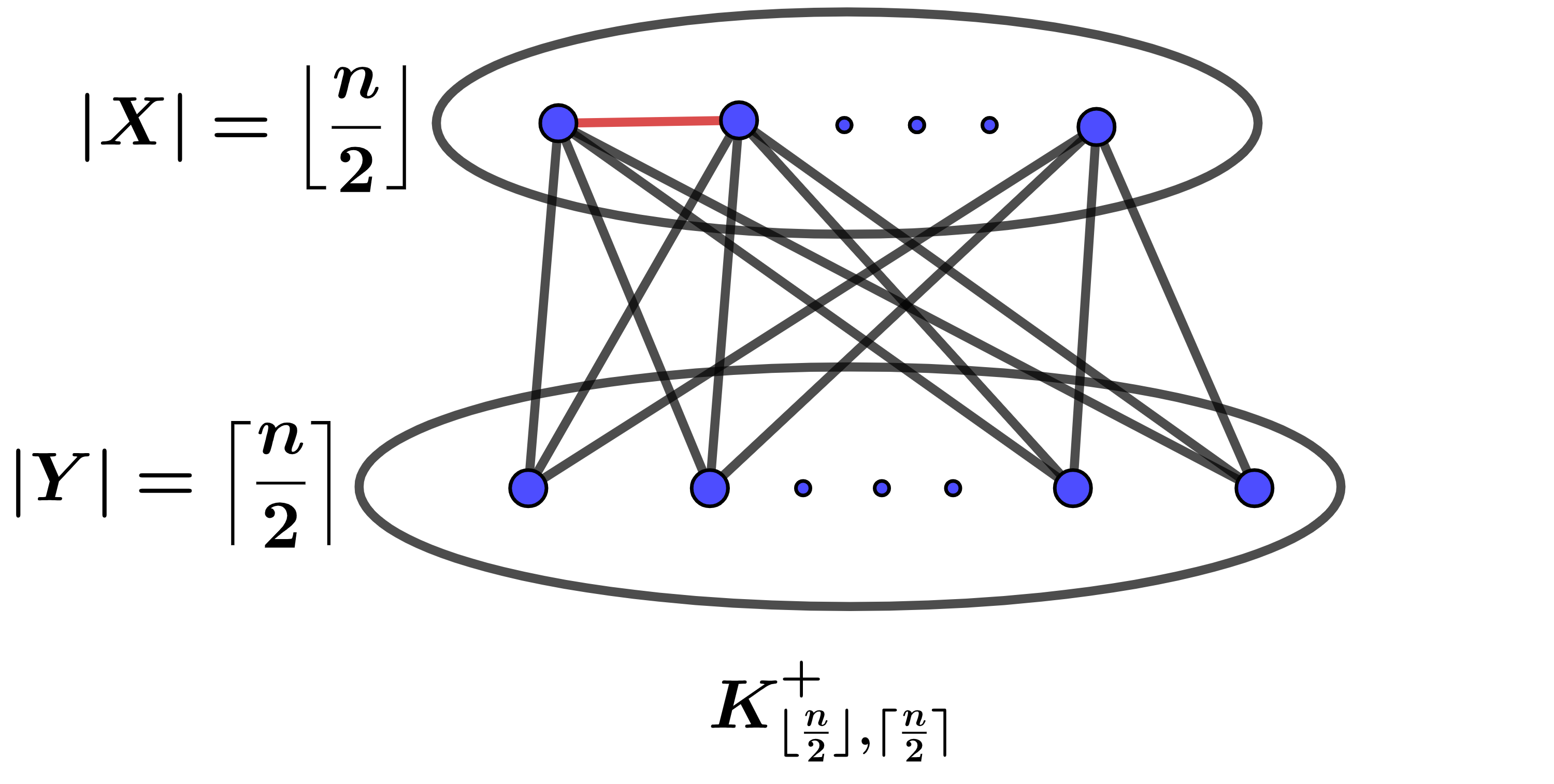}
\caption{The unique extremal graph in Theorem \ref{thm-n-F2}.}
\label{fig-2}
\end{figure}

One of the main results of this paper is as follows.

\begin{theorem} \label{thm-n-F2}
If $n\ge 7$ and $G$ is an $F_2$-free graph  on $n$ vertices, then
\[  \lambda (G) \le \lambda (K_{ \lfloor \frac{n}{2} \rfloor, \lceil \frac{n}{2} \rceil}^+), \]
equality holds if and only if $G=K_{ \lfloor \frac{n}{2} \rfloor, \lceil \frac{n}{2} \rceil}^+$.
\end{theorem}

\noindent
{\bf Remark.}
The bound $n\ge 7$ in Theorem \ref{thm-n-F2}
is the best possible, since for the case $n=6$,
there exists an $F_2$-free graph
with larger spectral radius. Namely,
$K_2\vee I_4$.
Upon computations, we have
$\lambda (K_2\vee I_4)\approx 3.722$,
while $\lambda (K_{3,3}^+)\approx 3.504$.

\medskip

\noindent
{\bf Observation.}
Using the fundamental inequality ${2e(G)}/{n} \le \lambda (G)$,
it is easy to verify that our result (Theorem \ref{thm-n-F2})
implies Theorem \ref{EFGG-F2}.
A tedious calculation gives $\lambda (K_{\frac{n}{2}, \frac{n}{2}}^+) <
\frac{n}{2} + \frac{2}{n} + \frac{8}{n^2}$ for even $n$,
and $\lambda(K_{\frac{n-1}{2}, \frac{n+1}{2}}^+) < \frac{n}{2} + \frac{3}{2n} + \frac{1}{n}$
for odd $n$ by using Lemma \ref{lem-21}.  Hence, we get $e(G) \le
\lfloor \frac{n}{2} \lambda (G) \rfloor
\le \lfloor \frac{n}{2} \lambda (K_{ \lfloor \frac{n}{2} \rfloor, \lceil \frac{n}{2} \rceil}^+)
\rfloor
= \lfloor n^2/4 \rfloor +1$.

\medskip 
We mention here that the proof of Theorem \ref{thm-n-F2} is completely different from that
of Theorem \ref{thmCFTZ20}, since the triangle removal lemma was used in the proof
of   Theorem \ref{thmCFTZ20} and  it requires the order $n$ of the graph $G$ to be sufficiently large.
The techniques in our proof are partially
inspired by recent articles of Zhai and Lin \cite{ZL2022jgt},
and Ning and Zhai \cite{NZ2021} as well.

\subsection{Spectral problems for graphs with given number of edges}

In 1970, Nosal \cite{Nosal1970}
determined the largest spectral radius of a  triangle-free graph
in terms of the number of edges, which states that
if $G$ is a triangle-free graph, then $\lambda (G)\le \sqrt{m}$.
Furthermore, Nikiforov \cite{Niki2002cpc,Niki2006laa} extended Nosal's theorem
and showed that if $G$ is triangle-free, then $\lambda (G) < \sqrt{m}$ unless $G$ is a complete bipartite graph (possibly with some isolated vertices).
In the sequel, when we consider the result on a graph with respect to
the given number of edges,
we shall ignore the possible isolated vertices if there are no confusions.

\begin{theorem}[Nosal--Nikiforov, \cite{Niki2002cpc,Niki2006laa}] \label{thmnosal}
If $G$ is a triangle-free graph with $m$ edges, then
\begin{equation}   \label{eq1}
\lambda (G)\le \sqrt{m} ,
\end{equation}
equality holds if and only if
$G$ is a complete bipartite graph.
\end{theorem}

Theorem \ref{thmnosal} implies that if  $G$ is a  bipartite graph
with $m$ edges,
then $  \lambda (G)\le \sqrt{m} $,
equality holds if and only if
$G$ is a complete bipartite graph.
On the one hand, inequality (\ref{eq1}) can imply
(\ref{eq-man}) in Theorem \ref{thmMan}.
Indeed, applying the Rayleigh inequality, we have
$\frac{2m}{n}\le \lambda (G)\le  \sqrt{m}$,
which yields $ m \le \lfloor {n^2}/{4} \rfloor$.
On the other hand,  combining  (\ref{eq1}) with  Mantel's theorem,  we obtain
$ \lambda (G)\le \sqrt{m}  \le \sqrt{\lfloor {n^2}/{4}\rfloor} =\lambda (K_{\lfloor \frac{n}{2}\rfloor, \lceil \frac{n}{2} \rceil })$.
Henceforth, inequality (\ref{eq1})  can also imply (\ref{eq2}) in Theorem \ref{thm-niki}.

\medskip
During the past few years, the inequality (\ref{eq1}) in Theorem \ref{thmnosal}
boosted the great interests of studying the maximum spectral radius of graphs
in terms of the number of edges, rather than the given number of vertices;
see \cite{Niki2002cpc} for an extension on $K_{r+1}$-free graphs,
\cite{Niki2009laa} for $C_4$-free graphs,
\cite{ZLS2021} for  $K_{2,r+1}$-free graphs,
and similar results of $C_5$-free and $C_6$-free graphs \cite{MLH2022} as well,
\cite{Niki2021} for an extension on book-free graphs,
and  \cite{LNW2021,ZS2022dm,LP2022oddcycle}  for non-bipartite triangle-free graphs. The spectral extremal problem for graphs with given size
is interesting in its own right, and it
is increasingly becoming an important and popular topic
in recent researches on spectral graph theory.

\medskip
In this section, we will focus the attention mainly on the extremal spectral problems
for $F_2$-free graphs with given number of edges.
Recall that $K_2\vee I_{\frac{m-1}{2}}$
denotes the join graph obtained from an edge $K_2$ and an independent set
$I_{\frac{m-1}{2}}$ such that each vertex of the edge is adjacent to each vertex of the
independent set.

\begin{figure}[H]
\centering
\includegraphics[scale=0.85]{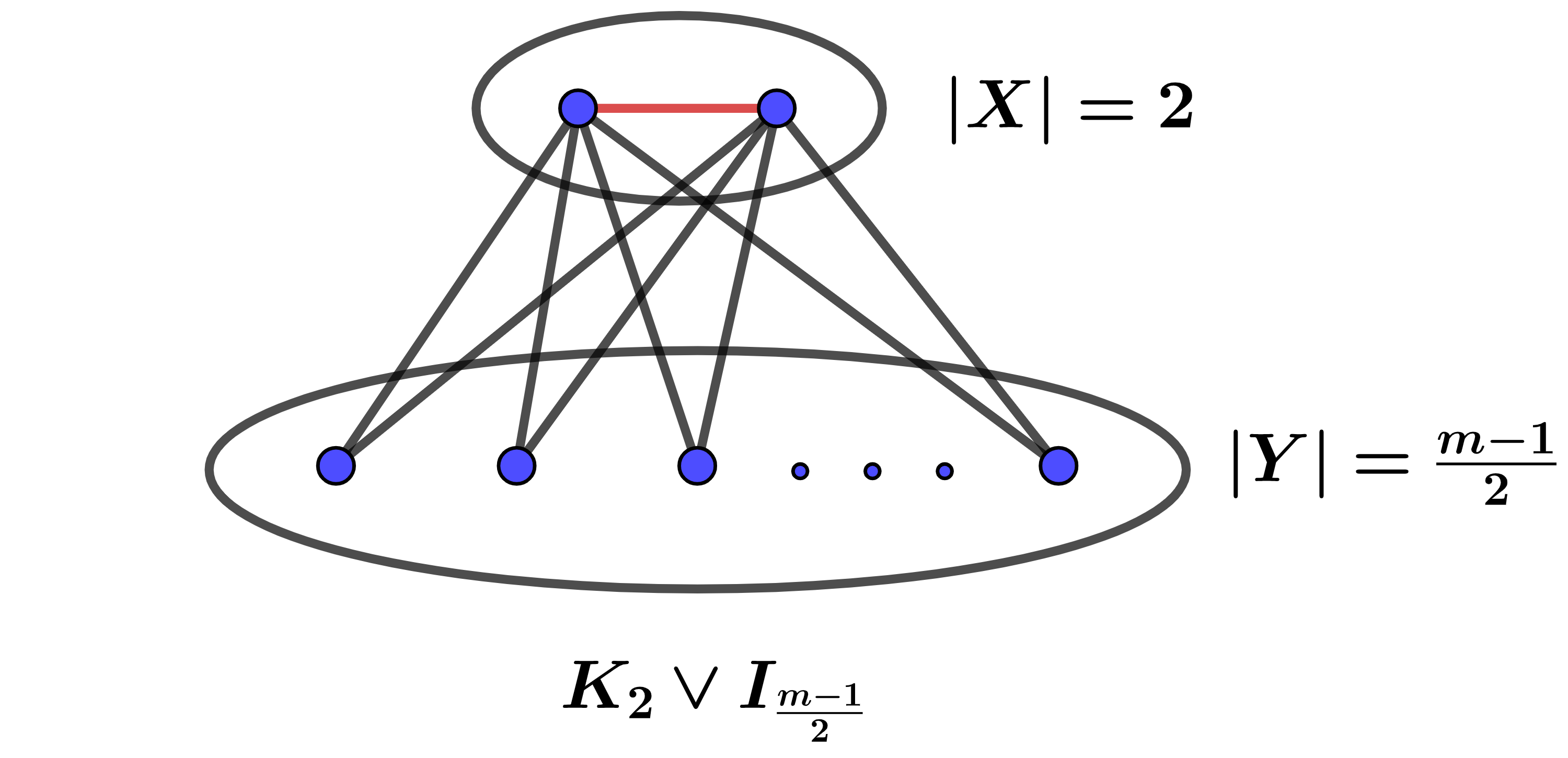}
\caption{The unique extremal graph in Theorem \ref{thm-m-F2}.}
\label{fig-m-edges}
\end{figure}

\begin{theorem}  \label{thm-m-F2}
If $m\ge 8$ and $G$ is an $F_2$-free graph with $m$ edges, then
\[  \lambda (G)\le \frac{1+\sqrt{4m-3}}{2}, \]
equality holds if and only if $G=K_2\vee I_{\frac{m-1}{2}}$.
\end{theorem}

\noindent
{\bf Remark.}
The bound $m\ge 8$ in Theorem \ref{thm-m-F2} is tight,
since for  $m=7$,
there exists an $F_2$-free graph
with larger spectral radius. Namely,
$G=K_1 \vee (K_3 \cup K_1)$.
An easy calculation gives $\lambda (G) \approx 3.086$,
but $\lambda (K_2 \vee I_3) = \frac{1+\sqrt{4m-3}}{2}=3$ is slightly smaller.

\medskip
At the end of this section,
we propose a conjecture for $F_k$-free graphs.

\begin{conjecture}[$F_k$-free graphs with given size]
\label{conj-m-Fk}
Let $k\ge 2$ be fixed and $m$ be large enough.
If $G$ is an $F_k$-free graph with $m$ edges, then
\[  \lambda (G)\le \frac{k-1 +\sqrt{4m -k^2+1}}{2}, \]
equality holds if and only if
$G=K_k \vee I_{\frac{1}{k}\left(m-{k \choose 2} \right)}$.
\end{conjecture}

\medskip
This paper is organized as follows.
In the next section, we will present the proof of Theorem \ref{thm-n-F2}.
In Section \ref{sec3},
the proof of Theorem \ref{thm-m-F2} will be provided.
In the last section, we conclude the paper
with some open problems.

\section{Proof of Theorem \ref{thm-n-F2}}
\label{sec2}

To begin with,
we shall present two lemmas for simplicity.
The first lemma gives the characterization
of the spectral radius of  $K_{ \lfloor \frac{n}{2} \rfloor, \lceil \frac{n}{2} \rceil}^+$, and it yields some lower/upper bounds for our purpose.
The second lemma states that if
one adds an edge within one part of a bipartite graph,
then the maximum spectral radius is attained
by adding this edge into the smaller vertex part
of a balanced bipartite graph.

\begin{lemma} \label{lem-21}
(a) If $n$ is even, then $\lambda(K_{\frac{n}{2}, \frac{n}{2}}^+)$
is the largest root of
\[  f(x)=x^3-x^2 - \frac{n^2}{4}x + \frac{n^2}{4}-n. \]
(b) if $n$ is odd, then $\lambda(K_{\frac{n-1}{2}, \frac{n+1}{2}}^+)$
is the largest root of
\[ g(x)=x^3 - x^2 + \frac{1-n^2}{4} x +
\frac{n^2}{4} - n - \frac{5}{4}. \]
Consequently, we have
\begin{equation} \label{eq-n2+2}
\lambda^2 (K_{ \lfloor \frac{n}{2} \rfloor, \lceil \frac{n}{2} \rceil}^+) >\left\lfloor \frac{n^2}{4} \right\rfloor +2.
\end{equation}
\end{lemma}

\begin{proof}
(a) Let $\bm{x}=(x_1,x_2,\ldots ,x_n)^T$ be a Perron eigenvector corresponding to $\lambda(K_{\frac{n}{2}, \frac{n}{2}}^+)$.
We partition the vertex set of $ K_{\frac{n}{2}, \frac{n}{2}}^+$ as $\Pi: V(\lambda(K_{\frac{n}{2}, \frac{n}{2}}^+))=
X_1 \cup X_2 \cup Y$, where $X_1=\{u,v\}$ forms an edge,
$X_1\cup X_2$ and $Y$ are partite sets of $K_{\frac{n}{2}, \frac{n}{2}}$.
By comparing the neighborhoods, we can see that $x_u=x_v$,
all coordinates of the vector $\bm{x}$
corresponding to vertices of $X_2$  are equal.
(the coordinates of vertices of $Y$ are equal)
For notational convenience, we may assume that
$x_u=x_v=x$, $x_s=y$ for all $s\in X_2$ and $x_t=z$
for all $t\in Y$. Then
\[  \begin{cases}
\lambda x= x + \frac{n}{2} z, \\
\lambda y = \frac{n}{2}z, \\
\lambda z = 2x + (\frac{n}{2}-2)y.
\end{cases} \]
Thus $\lambda(K_{\frac{n}{2}, \frac{n}{2}}^+)$ is the largest eigenvalue of
\[  B_{\Pi} = \begin{matrix} X_1 \\ X_2 \\ Y
\end{matrix} \begin{bmatrix}
1 & 0 & \frac{n}{2} \\
0 & 0 & \frac{n}{2} \\
2 & \frac{n}{2}-2 & 0
\end{bmatrix}.  \]
By calculation, we know that $\lambda(K_{\frac{n}{2}, \frac{n}{2}}^+)$ is the largest  root of
\[ f(x)= \det (xI_3 - B_{\Pi})=x^3-x^2 - \frac{n^2}{4}x + \frac{n^2}{4}-n.  \]
(b) For odd $n$, the proof is similar with the previous case.
We partition the vertex set of $K_{\frac{n-1}{2}, \frac{n+1}{2}}^+$
as $\Pi': V(K_{\frac{n-1}{2}, \frac{n+1}{2}}^+) = X_1\cup X_2\cup Y$,
where $X_1=\{u,v\}$ forms an edge, $X_1\cup X_2$ and $Y$
are partite sets of $K_{\frac{n-1}{2}, \frac{n+1}{2}}$
satisfying $|X_1| + |X_2|= \frac{n-1}{2}$
and $|Y|=\frac{n+1}{2}$. Then
a similar argument yields that
$\lambda (K_{\frac{n-1}{2}, \frac{n+1}{2}}^+)$ is the largest eigenvalue of
\[  B_{\Pi'}= \begin{matrix} X_1 \\ X_2 \\ Y
\end{matrix} \begin{bmatrix}
1 & 0 & \frac{n+1}{2} \\
0 & 0 & \frac{n+1}{2} \\
2 & \frac{n-1}{2}-2 & 0
\end{bmatrix}.  \]
Thus, $\lambda (K_{\frac{n-1}{2}, \frac{n+1}{2}}^+)$
is the largest root of
\[  g(x)= \det (xI_3- B_{\Pi'})=x^3 - x^2 + \frac{1-n^2}{4} x +
\frac{n^2}{4} - n - \frac{5}{4}.  \]
Finally, we are ready to prove (\ref{eq-n2+2}).
By calculation, we can verify that
\[  f(\sqrt{{n^2}/{4 }+2}) = \sqrt{n^2+8} -n -2<0, \]
and
\[  g(\sqrt{{(n^2-1)}/{4} +2}) = \sqrt{n^2+7} -n-3 <0. \]
This completes the proof.
\end{proof}

In fact, Rayleigh's formula also gives a lower bound
on the spectral radius.
\begin{equation} \label{eq-3}
\lambda (K_{ \lfloor \frac{n}{2} \rfloor, \lceil \frac{n}{2} \rceil}^+)
> \frac{2e(K_{ \lfloor \frac{n}{2} \rfloor, \lceil \frac{n}{2} \rceil}^+)}{n}
= \begin{cases}
\frac{n}{2} + \frac{2}{n}, & \text{if $n$ is even;}\\
\frac{n}{2} + \frac{3}{2n}, & \text{if $n$ is odd.}
\end{cases}
\end{equation}
Therefore, for even $n$, we get from (\ref{eq-3})
that $\lambda^2 > (\frac{n}{2} + \frac{2}{n})^2 >
\frac{n^2}{4} +2$.
On the other hand, for odd $n$, we construct a unit vector
$\bm{x}:=(\frac{1}{\sqrt{n-1}}, \ldots ,\frac{1}{\sqrt{n-1}}, \frac{1}{\sqrt{n+1}}, \ldots ,\frac{1}{\sqrt{n+1}})^T$, where
$\frac{1}{\sqrt{n-1}}$ corresponds to vertices of $X$, and
$\frac{1}{\sqrt{n+1}}$ corresponds to vertices of $Y$, respectively
(see Figure \ref{fig-2}).  Then Rayleigh's formula implies
$\lambda (K_{\frac{n-1}{2}, \frac{n+1}{2}}^+)
\ge \bm{x}^TA(K_{\frac{n-1}{2}, \frac{n+1}{2}}^+) \bm{x} =
\frac{\sqrt{n^2-1}}{2} +
\frac{2}{n-1}$, which yields $\lambda^2 (K_{\frac{n-1}{2}, \frac{n+1}{2}}^+) \ge (\frac{\sqrt{n^2-1}}{2} +
\frac{2}{n-1})^2 > \frac{n^2-1}{4} +2$.
The above argument provides  an alternative proof of (\ref{eq-n2+2}).

\begin{lemma} \label{lem-22}
Let $a,b\ge 2$ be integers with $a+b=n$,
and $K_{a,b}^+$ be the graph obtained from
the complete bipartite graph $K_{a,b}$
by adding an edge in the part of size $a$. Then
\[  \lambda (K_{a,b}^+) \le \lambda (K_{ \lfloor \frac{n}{2} \rfloor, \lceil \frac{n}{2} \rceil}^+), \]
equality holds if and only if $a=\lfloor \frac{n}{2} \rfloor$
and $b=\lceil \frac{n}{2} \rceil$.
\end{lemma}

\begin{proof}
Similar with the proof of Lemma \ref{lem-21}, we know that
$\lambda (K_{a,b}^+)$ is the largest root of
\[  f_{a,b}(x)=x^3-x^2 -abx +ab -2b. \]
It is sufficient to prove  the following two cases.

{\bf Case 1.}
If $a\ge b+1$, then
our goal is to prove that
$\lambda (K_{a,b}^+) < \lambda (K_{a-1,b+1}^+)$.
By computation, we obtain
\[ f_{a,b}(x)- f_{a-1,b+1}(x)= x(a-b-1) -a+b+3.  \]
If $a=b+1$, then $f_{a,b}(x)=f_{a-1,b+1}(x)+2$,
and $f_{a-1,b+1}(x) < f_{a,b}(x)$ for every $x \in \mathbb{R}$;
if $a\ge b+2$,
then we get $f_{a-1,b+1}(x) < f_{a,b}(x)$ for every $x> \frac{a-b-3}{a-b-1}$.
Note that $\lambda(K_{a,b}^+)\ge \lambda (K_3)=2$.
Therefore, we obtain $f_{a-1,b+1}(\lambda(K_{a,b}^+)) <
f_{a,b}(\lambda(K_{a,b}^+))=0$.
Recall that $\lambda(K_{a-1,b+1}^+)$
is the largest root of $f_{a-1,b+1}(x)$. Thus
we get $\lambda(K_{a,b}^+) < \lambda(K_{a-1,b+1}^+)$.

{\bf Case 2.}
If $a+2\le b$, then we need to prove $\lambda (K_{a,b}^+) <
\lambda (K_{a+1,b-1}^+)$.
Similarly, we have
\[  f_{a,b}(x)- f_{a+1,b-1}(x) = x(b-a-1) + a-b-1. \]
In view of $b\ge a+2$,  for every $x>\frac{b-a+1}{b-a-1}$,
we get $f_{a+1,b-1}(x) < f_{a,b}(x)$.
Since $a\ge 2$ and $b\ge 4$, we can see that
$K_2\vee I_4$ is a subgraph of $K_{a,b}^+$.
Then $\lambda (K_{a,b}^+) \ge \lambda (K_2\vee I_4)
\approx 3.372$. Since $\frac{b-a+1}{b-a-1}= 1+ \frac{2}{b-a-1}\le 3$,
we obtain $f_{a+1,b-1}(\lambda (K_{a,b}^+)) < f_{a,b}(\lambda (K_{a,b}^+)) =0$, which  leads to
$\lambda (K_{a,b}^+) < \lambda (K_{a+1,b-1}^+)$, as desired.
\end{proof}

Now, we are in a position to prove Theorem \ref{thm-n-F2}.

\begin{proof}[{\bf Proof of Theorem \ref{thm-n-F2}}]
Let $G$ be an $F_2$-free graph on $n\ge 7$ vertices
with maximum spectral radius.
Then $\lambda (G)\ge \lambda (K_{ \lfloor \frac{n}{2} \rfloor, \lceil \frac{n}{2} \rceil}^+)$.
The aim of the proof is to show that
$G=K_{ \lfloor \frac{n}{2} \rfloor, \lceil \frac{n}{2} \rceil}^+$.
Clearly, one can see that $G$ is connected. Otherwise,
choosing $G_1$ and $G_2$
as two components with $\lambda(G_1) = \lambda (G)$,
and adding an edge between $G_1$ and $G_2$, we get a new graph
on $n$ vertices, which is still $F_2$-free and has larger spectral radius than
$G$, a contradiction.
By Lemma \ref{lem-21}, 
\begin{equation}  \label{eq-lower}
\lambda^2(G) > \lfloor n^2/4\rfloor +2.
\end{equation}
Let $X=(x_1,x_2,\ldots ,x_n)^T$ be a Perron vector
of $G$ and $u\in V(G)$ be a vertex such that $x_u=\max\{x_v: v\in V(G)\}$.
For notational convenience, we denote
$\lambda =\lambda (G)$, $A=N_G(u)$ and $B=
V(G) \setminus (A\cup \{u\})$. Then $|A| + |B| +1=n$.

To state the proof in details,
we outline the main steps as follows.
First of all, we shall show that $G[A]$ is a star with center $u_0$, and $G[B]$ is empty,
and the edges between $A\setminus \{u_0\}$
and $B$ form a complete bipartite graph.
Secondly, we will prove that
$|A|=\lceil \frac{n}{2}\rceil +1$ and $|B|=\lfloor \frac{n}{2}\rfloor -2$.
We write $e(A)$ for the number of edges
with two endpoints in $A$,
and $e(A,B)$ for the number of edges with one endpoint in $A$
and the another in $B$; see Figure \ref{fig-structure}.

\begin{figure}[H]
\centering
\includegraphics[scale=0.85]{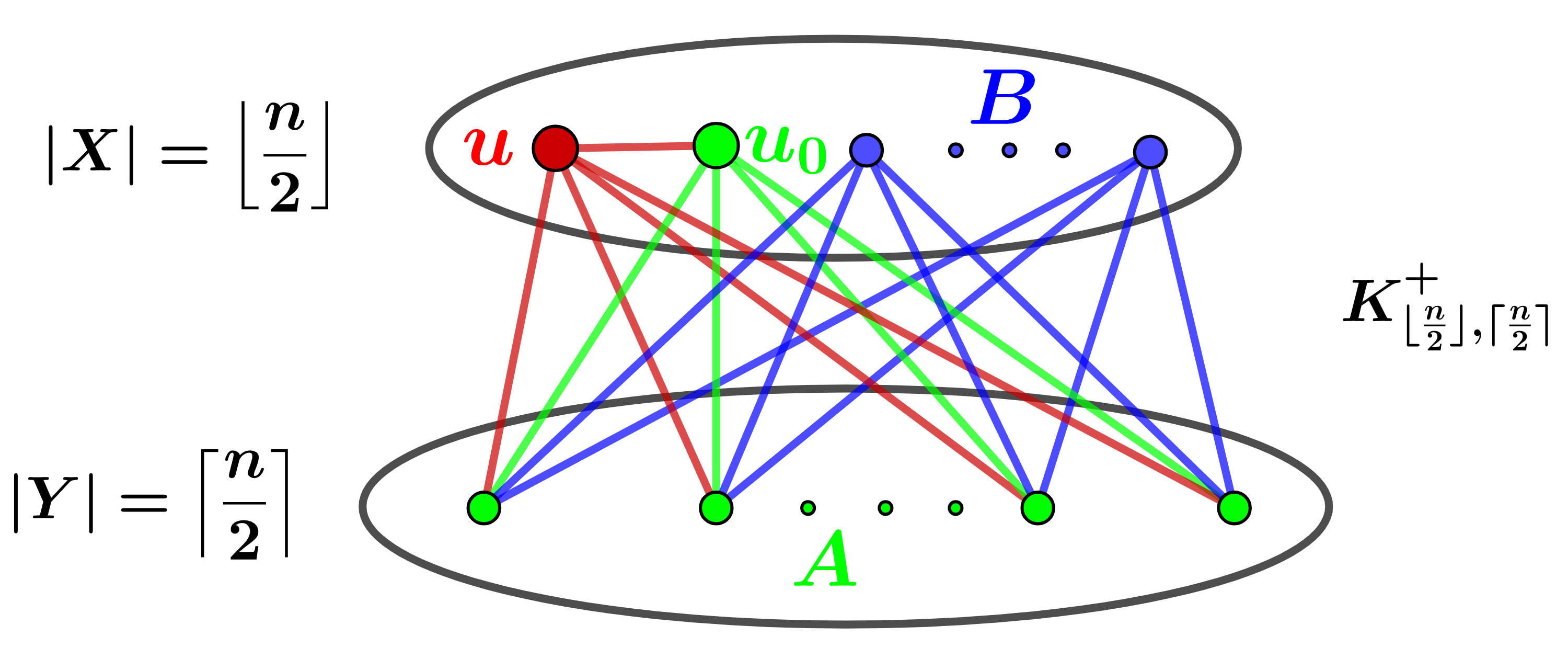}
\caption{The structure of the extremal graph.}
\label{fig-structure}
\end{figure}

\medskip
{\bf Claim 1. } {\it $|A| \ge \lambda$. Furthermore,
we have $|A| \ge \lceil \frac{n}{2}\rceil$ and $|B| \le \lfloor \frac{n}{2}\rfloor -1$. }

\begin{proof}[Proof of Claim 1]
By the maximality of  $x_u$,
we get $\lambda x_u = \sum_{v\in A} x_v \le |A| x_u$.
Hence $\lambda \le |A|$. By inequality (\ref{eq-3}), we obtain
$\lambda > \frac{n}{2} + \frac{3}{2n}$. As $|A|$ is an integer,
we have $|A|\ge \lceil \frac{n}{2}\rceil$.
Keeping in mind that $|A| + |B| =n-1$. Then $|B| \le \lfloor \frac{n}{2}\rfloor -1$.
\end{proof}

\medskip
{\bf Claim 2. } {\it $e(A)\ge 1$ and $|B|\ge 1$.}

\begin{proof}[Proof of Claim 2]
Suppose on the contrary that $e(A)=0$.
Denote by $d_A(v)$ the number of neighbors of $v$ in
vertex set $A$. Namely, $d_A(v)=|N_G(v) \cap A|$.
It is known that
\begin{equation}  \label{eq-lambda-2}
\lambda^2 x_u = \sum_{v\in A} \sum_{w\in N(v)} x_w
= |A|x_u + \sum_{v\in A} d_A(v)x_v + \sum_{w\in B} d_A(w)x_w.
\end{equation}
Roughly speaking, the equality (\ref{eq-lambda-2})
provides a nice estimate of $\lambda (G)$
for nearly-regular graphs, since all coordinates of the
Perron vector of such graphs are almost equal.
Observe that $\sum_{w\in B} d_A(w)= e(A,B) \le |A||B|$. Then
(\ref{eq-lambda-2}) yields
\[ \lambda^2 x_u =|A| x_u + \sum_{w\in B} d_A(w)x_w
\le |A| (1+ |B|)x_u .  \]
Note that $|A| + |B| +1=n$. Then AM-GM inequality gives 
$\lambda^2 \le \lfloor n^2/4 \rfloor$,
which contradicts with (\ref{eq-lower}).
Thus, we have proved $e(A)\ge 1$.

Suppose on the contrary that $|B|=0$.
Then $|A|=n-1$.
Since $G$ is $F_2$-free, we see that $G[A]$ has no $2K_2$,
and hence $G[A]$ is $P_4$-free (the path on $4$ vertices).
Except for the isolated vertices, the induced subgraph $G[A]$ is
connected. Hence either $G[A]=K_3\cup I_{|A|-3}$
or $G[A]=K_{1,t}
\cup I_{|A|-t-1}$ for some $t\ge 1$.
In the first case, we obtain
$\lambda \le \lambda (K_1\vee (K_3 \cup I_{n-4}))$,
which is smaller than $\frac{n}{2}$. Indeed, we know that
$\lambda (K_1\vee (K_3 \cup I_{n-4}))$ is the largest root of
\[  p(x)= x^3-2x^2 + (1-n)x +2n-8. \]
It follows that $p(\frac{n}{2}) = \frac{n^3}{8} - n^2 + \frac{5 n}{2} - 8>0$
for every $n\ge 7$. Moreover,
we have $p'(x)=3x^2 -4x + (1-n)>0$
for every $x\ge \frac{n}{2}$.
Therefore, we have $p(x)>0$
for every $x\ge \frac{n}{2}$, which implies
$\lambda (K_1\vee (K_3 \cup I_{n-4})) < \frac{n}{2}$, as needed.
In the second case,
we have $\lambda \le \lambda (K_2\vee I_{n-2})
=  \frac{1}{2}(1 + \sqrt{8n-15}) \le \frac{n}{2} <
\lambda (K_{ \lfloor \frac{n}{2} \rfloor, \lceil \frac{n}{2} \rceil}^+)$ for every $n\ge 8$, a contradiction.
For the case $n=7$,
a direct computation gives
$\lambda (K_2\vee I_5) \approx 3.701$
and $\lambda (K_{3,4}^+) \approx 3.848$,
so we get $\lambda \le \lambda (K_2\vee I_5) < \lambda (K_{3,4}^+)$.
This is also a contradiction.
\end{proof}

Next, we partition the remaining proof in two cases.

\medskip
{\bf Case 1.} First of all, we consider the case $G[A]=K_3\cup I_{|A| -3}$.
Denote by $\{u_1,u_2,u_3\}$ the vertex of the copy of $K_3$.
Every vertex of $B$ has at most one neighbor in $\{u_1,u_2,u_3\}$.
Then $e(A,B) \le |B| (|A|-2)$ and
\[  \lambda^2 x_u \le
(|A| + 2e(A) + e(A,B))x_u \le
(|A| (|B|+1) - 2|B| +6)x_u.  \]
By computation, the maximum of $|A| (|B|+1) - 2|B|$
attains at $|B|= \lfloor \frac{n}{2}\rfloor -2$.
Thus, we get $\lambda \le (\lceil \frac{n}{2}\rceil +1)(\lfloor \frac{n}{2}\rfloor -1) -2 (\lfloor \frac{n}{2}\rfloor -2) +6$,
which yields
$\lambda \le \lfloor \frac{n^2}{4}\rfloor +1$ for every $n\ge 8$,
and  $\lambda \le \lfloor \frac{n^2}{4}\rfloor +2$ for $n=7$.
This is a contradiction with (\ref{eq-lower}).

{\bf Case 2.} Now, we deal with the remaining case $G[A]=K_{1,t}
\cup I_{|A| -t-1}$
for some $t\ge 1$.
Denote by $u_0$ the central vertex of the star $K_{1,t}$ in $G[A]$,
and $\{u_1,u_2,\ldots ,u_t\}$ the leaves of $K_{1,t}$.
If $t=1$, then $e(A,B)\le |A| |B|$ and 
$\lambda^2 \le |A| +2t + e(A,B) \le 
|A|(1+|B|)+2 \le \lfloor n^2/4 \rfloor +2$, 
which contradicts with (\ref{eq-lower}). 
In what follows, we will consider the case $t\ge 2$ and 
proceed the proof in two subcases.

{\bf Subcase 2.1.}
Suppose that $u_0$ has a neighbor in $B$, say $w_0\in B$.
Since $t\ge 2$ and $G$ is $F_2$-free, we observe that $w_0u_i \notin E(G)$ for each $i\in \{1,2,\ldots ,t\}$.
In addition, every vertex  $w\in B$ satisfies $d_A(w) \le |A| -1$.
Then $e(A,B) \le (|A| -t) + (|A|-1) (|B| -1)$ and
\[  \lambda^2 \le |A| +2t +e(A,B)
= (|A| -1)(|B| +1) +2 +t. \]
Note that $t \le |A| -1$ and $|A| + |B| +1=n$.
Thus, we get
$\lambda^2 \le (|A| -1)(|B|+2) +2
\le \lfloor {n^2}/{4} \rfloor +2$, a contradiction.

{\bf Subcase 2.2.}
Finally, we consider the case $d_B(u_0)=0$.
In this case,
we claim that $e(B)=0$.
Otherwise, if $ww'$ is an edge in $G[B]$,
then for each $i\in \{1,2,\ldots ,t\}$,
the vertex $u_i$ has at most one neighbor in $\{w,w'\}$,
and thus $d_B(u_i) \le |B|-1$.
Then $e(A,B) \le (|A| -t-1) |B| + t (|B|-1)= |A||B| - |B| -t$.
Combining with $t\le |A| -1$, we obtain
\[  \lambda^2 \le |A| + 2t +e(A,B) \le
|A|  + |A||B| -|B| +t \le (|A| -1) (|B| +2) +1. \]
Due to $|A| + |B| +1=n$, we get
$\lambda^2 \le
\lfloor n^2/4\rfloor +1$,
which is a contradiction. Therefore, we have proved that $e(B)=0$.

Since $G$ has the maximum spectral radius.
Then for each $w\in B$, we have $N_A(w)= A \setminus \{u_0\}$.
Furthermore, the maximality also implies that $t=|A|-1$.
Therefore, we obtain $V(G)=\{u\} \cup A \cup B$,
where $A=\{u_0,u_1,\ldots ,u_t\}$ forms a star
and $B$ is an independent set.
Moreover,  $A \setminus \{u_0\}$ and $B$ form a complete bipartite graph.
Putting $\{u,u_0\}$ and $B$ together,
the graph $G$ can be obtained from
the complete bipartite graph $K_{|B|+2,|A|-1}$
by adding an edge into the part of size $|B|+2$.
Recall that $|A| \ge \lceil \frac{n}{2}\rceil$ and $|B| \le \lfloor \frac{n}{2}\rfloor -1$.
By Lemma \ref{lem-22},
we can see that when $|A| = \lceil \frac{n}{2}\rceil +1$,
the graph $G$ attains maximum spectral radius,
and then we get $G=K_{ \lfloor \frac{n}{2} \rfloor, \lceil \frac{n}{2} \rceil}^+$, as required.
\end{proof}

\section{Proof of Theorem  \ref{thm-m-F2}}
\label{sec3}

Before showing the proof of Theorem  \ref{thm-m-F2},
we start with the following lemma.

\begin{figure}[H]
\centering
\includegraphics[scale=0.85]{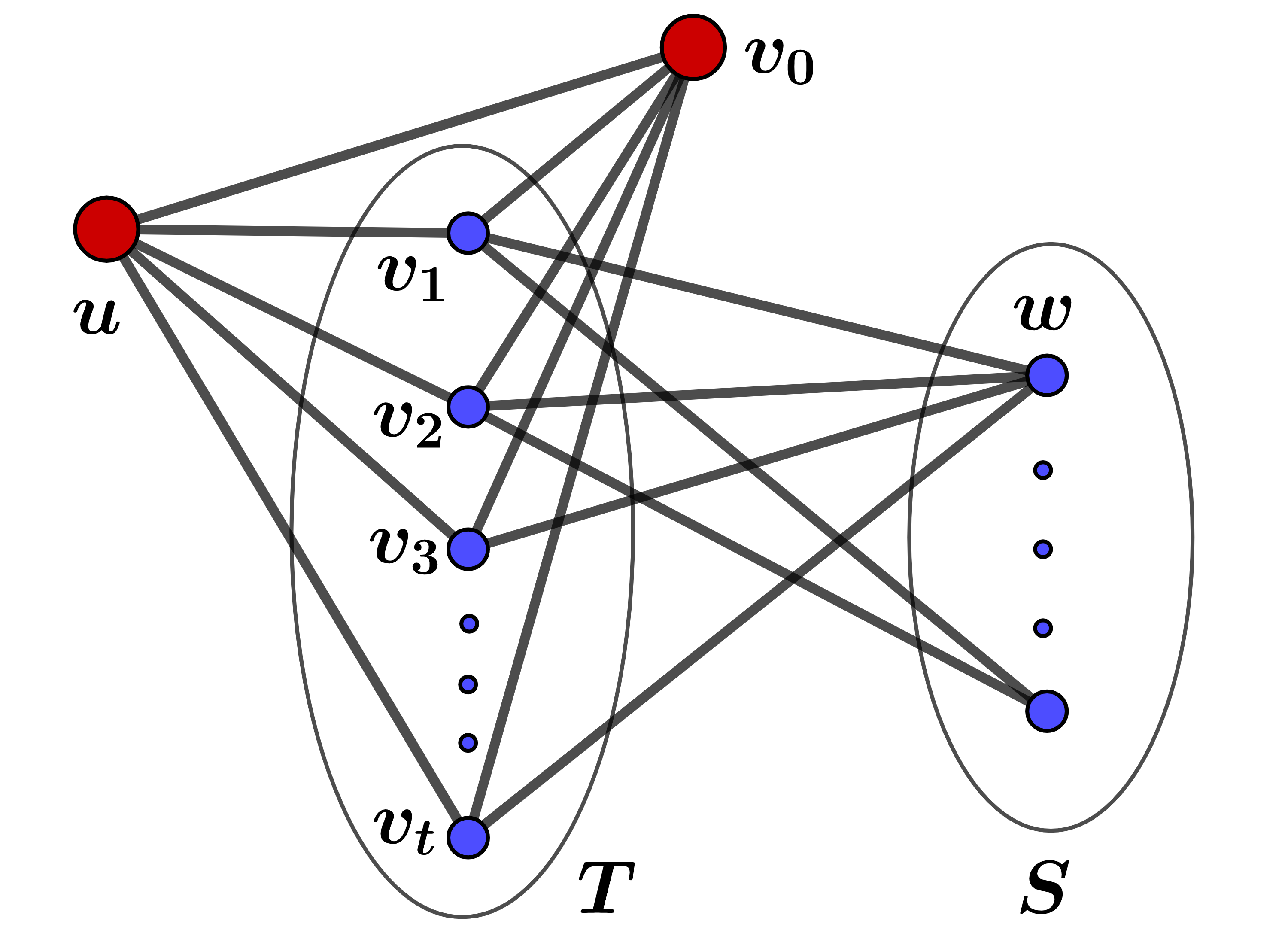}
\caption{An approximate structure of the graph  $G$.}
\label{fig-4}
\end{figure}

\begin{lemma} \label{lem31}
Let $G$ be an $m$-edge graph obtained from a bipartite graph
$B_{s,t}$ by adding two new vertices $u$ and $v_0$,
an edge $uv_0$ and all edges between $\{u,v_0\}$ and $\{v_1,\ldots ,v_t\}$. Then
\[ \lambda (G)< \frac{1+ \sqrt{4m-3}}{2}.  \]
\end{lemma}

For convenience of readers, we draw the graph $G$ in Figure \ref{fig-4}.

\begin{proof}
Since the structure of $B_{s,t}$ is stochastic,
it seems impossible to get the required bound by a straightforward calculation.
The key idea of the proof depends on a trick analysis
of the Perron eigenvector of $\lambda (G)$.
Since $u$ and $v_0$ has the same neighborhood in $G$,
we have $x_u = x_{v_0}$.
In what follows, we will establish
an upper bound on $\lambda^2 x_u$ and
then show that $\lambda (G) <
\frac{1+\sqrt{4m-3}}{2}$. Since $\lambda x_u = x_{v_0} +
\sum_{i=1}^t x_{v_i}$, we get
\[  \sum_{i=1}^t x_{v_i} = (\lambda -1)x_u. \]
Moreover, we have
\begin{align*}
\lambda^2 x_u &= \lambda x_{v_0} + \sum_{i=1}^t
\lambda x_{v_i}
= x_u + \sum_{i=1}^t x_{v_i} +
\sum_{i=1}^t \left(x_u + x_{v_0} + \sum_{w\in N(v_i)\cap B} x_w\right) \\
&= (2t+1) x_u + \sum_{i=1}^t x_{v_i} + \sum_{w\in B} d_A(w) x_w.
\end{align*}
Observe that $\lambda x_w= \sum_{v\in N(w)} x_v
\le \sum_{i=1}^t x_{v_i}$,
we get $x_w \le \frac{1}{\lambda} \sum_{i=1}^t x_{v_i}$. Then
\begin{align*}
\lambda^2 x_u &\le (2t+1) x_u +
\left( 1+ \frac{1}{\lambda} \sum_{w\in B} d_A(w) \right)
\left( \sum_{i=1}^t x_{v_i} \right) \\
&= (2t+1) x_u + \left(1+\frac{m-2t-1}{\lambda} \right) (\lambda -1)x_u \\
&< m x_u + (\lambda -1)x_u.
\end{align*}
Thus, we get $\lambda^2- \lambda - (m-1) <0$,
which yields $\lambda < \frac{1+\sqrt{4m-3}}{2}$.
\end{proof}

In the sequel,
we will provide the proof of Theorem \ref{thm-m-F2}.

\begin{proof}[{\bf Proof of Theorem \ref{thm-m-F2}}]
Let $G$ be an $F_2$-free graph with $m\ge 8$ edges
and maximum spectral radius.
We may assume that $G$ satisfies $\lambda (G) \ge \frac{1+\sqrt{4m-3}}{2} $.
Our goal is to show that $G=K_2\vee I_{\frac{m-1}{2}}$.
Let $\bm{x}=(x_1,x_2,\ldots ,x_n)^T$ be the unit Perron vector
of $G$ and $u\in V(G)$ be a vertex such that
$x_u =\max \{x_v : v\in V(G)\}$.
Moreover, let $A=N_G(u)$ and $B=V(G) \setminus (\{u\}\cup A)$.

In our proof, we will frequently use two facts.
The first fact admits that $G$ is connected.
Otherwise, we can choose $G_1$ and $G_2$ as two different components,
where $G_1$ attains the spectral radius of $G$,
deleting an edge within $G_2$, and then adding an edge between $G_1$ and $G_2$, we get a new $m$-edge graph which is still $F_2$-free and has larger spectral radius,
which is a contradiction.

The second fact asserts that $d(w) \ge 2$ for each vertex $w\in B$.
Otherwise, if $w\in B$ has degree $1$, say $wv\in E(G)$,
where $v\in V(G)$ and $v\neq u$,
then replacing the edge $wv$ with $wu$,
we get a new $m$-edge $F_2$-free graph $G'$ with larger spectral radius.
Indeed, since $x_v\le x_u$, we have
\begin{equation}  \label{degree-two}
\lambda (G') \ge \bm{x}^TA(G') \bm{x} = 2\left(\sum_{\{i,j\}\in E(G)}
x_ix_j \right)- 2x_wx_v + 2x_wx_u \ge \lambda (G).
\end{equation}
In fact, we claim that $\lambda (G') > \lambda (G)$.
Assume on the contrary that
$\lambda (G')= \lambda(G)$. Then all inequalities in
(\ref{degree-two}) become equalities,
and $\bm{x}$ is also a unit eigenvector of $\lambda (G')$.
Namely, $A(G')\bm{x}=\lambda (G') \bm{x} = \lambda (G) \bm{x}$.
Consider the vertex $u$, we observe that
$\lambda (G')x_u = \sum_{v\in N_G(u)} x_v + x_w >
\sum_{v\in N_G(u)} x_v = \lambda (G)x_u$.
Consequently, we get $\lambda (G') > \lambda (G)$,
which contradicts with our assumption.

\medskip
{\bf Claim 1. } {\it $e(A)\ge 1$.}

\begin{proof}[Proof of Claim 1]
Assume on the contrary that $e(A)=0$.  Thus,
\begin{align*}
\lambda^2 x_u = \sum_{v\in A} \sum_{w\in N(v)} x_w
&= |A|x_u + \sum_{v\in A} d_A(v) x_v + \sum_{w\in B} d_A(w) x_w \\
&\le (|A| + e(A,B))x_u \le m x_u,
\end{align*}
where the last inequality holds since
$m=e(G)=|A| + e(A,B) + e(B)$.
It follows that $\lambda \le \sqrt{m}$.
Observe that $\lambda (G) \ge \frac{1+\sqrt{4m-3}}{2}  > \sqrt{m}$. This leads to a contradiction.
\end{proof}

\medskip
{\bf Claim 2. } {\it $G[A]=K_3 \cup I_{|A|-3}$
or $G[A]=K_{1,t}\cup I_{|A|-t-1}$ for some $t\ge 1$.}

\begin{proof}[Proof of Claim 2]
Since $G$ is $F_2$-free and $e(A)\ge 1$ by Claim 1,
we can see that
the induced subgraph $G[A]$ is $2K_2$-free, so it
consists of only one connected
component and some isolated vertices.
If $e(A)=1$, then it is clear that $G[A]=K_{1,1} \cup I_{|A|-2}$.
For the case $e(A)\ge 2$,
observe that the longest path in $G[A]$ is
$P_3$ on $3$ vertices,
the only component in $G[A]$ is either the cycle $K_3$,
or the star $K_{1,t}$ for some $t\ge 2$.
\end{proof}

We denote $A_+=\{v\in A: d_A(v) \ge 1\}$
and $A_0= A \setminus A_+$.
In other words, $A_0$ is the set of isolated vertices
in the induced subgraph $G[A]$.
In what follows, we present the proof in two cases.

\medskip
{\bf Case 1.} $B=\varnothing$. In this case,
we have $\lambda x_u = \sum_{v\in A} x_v$ and
\[  \lambda^2x_u = \sum_{v\in A} \lambda x_v = |A| x_u + \sum_{v\in A} d_A(v)x_v. \]
Thus, we get
\[ (\lambda^2 - \lambda )x_u = |A|x_u +
\sum_{v\in A_+} (d_A(v) -1)x_v - \sum_{v\in A_0} x_v. \]
Note that $x_u$ is the maximum coordinate in the Perron vector. It follows that
\[ \lambda^2 - \lambda \le |A| + 2e(A_+) - |A_+|
- \sum_{v\in A_0} \frac{x_v}{x_u} . \]
Notice that $\lambda \ge \frac{1+\sqrt{4m-3}}{2}$,
which implies $\lambda^2 - \lambda \ge m-1
= |A| +e(A_+) -1$. Then
\begin{equation} \label{eq5}
\sum_{v\in A_0} \frac{x_v}{x_u} \le
e(A_+) - |A_+| +1.
\end{equation}
By Claim 2, if $G[A]=K_3 \cup I_{|A| -3}$,
then $|A|=m-3$ and $G=K_1\vee (K_3 \cup I_{m-6})$.
By computation, we know that $\lambda (G)$
is the largest root of
\[ h(x)=x^3 - 2x^2 + 3x -mx +2m-12.  \]
It is easy to check that $h(\frac{1+\sqrt{4m-3}}{2})
>0$ for every integer $m\ge 8$.
Additionally, we have $h'(x)=3x^2 -4x +3-m >0$
for every $x\ge \frac{1+\sqrt{4m-3}}{2}$.
Hence, we obtain $\lambda (G) < \frac{1+\sqrt{4m-3}}{2}$,
a contradiction. Therefore, Claim 2 implies
$G[A]=K_{1,t} \cup I_{|A|-t-1}$ for some $t\ge 1$.
Then we get $e(A_+)- |A_+| =-1$.
Combining with (\ref{eq5}),
we have  $A_0= \varnothing$ and $t=|A|-1=\frac{m-1}{2} $,
then $G=K_1\vee K_{1,t}= K_2 \vee I_{\frac{m-1}{2}}$,
which is the desired extremal graph.

{\bf Case 2.} $B\neq \varnothing$.
Since $\lambda^2 x_u = |A| x_u +
\sum_{v\in A} d_A(v)x_v + \sum_{w\in B} d_A(w) x_w$,
we get
\[  (\lambda^2 - \lambda ) x_u
= |A| x_u + \sum_{v\in A_+} (d_A(v) -1)x_v - \sum_{v\in A_0} x_v + \sum_{w\in B} d_A(w) x_w.  \]
Invoking the fact $\sum_{w\in B} d_A(w) x_w \le e(A,B)x_u$, we have
\[  \lambda^2- \lambda \le |A| + 2e(A_+) - |A_+|
- \sum_{v\in A_0} \frac{x_v}{x_u} + e(A,B). \]
Note that $\lambda^2 - \lambda \ge m-1
= |A| + e(A_+) + e(A,B) +e(B) -1$. Then
\begin{equation} \label{eq-6}
e(B)\le e(A_+) - |A_+| - \sum_{v\in A_0} \frac{x_v}{x_u} +1.
\end{equation}
Next, we partition the remaining proof in two subcases.

{\bf Subcase 2.1.}
If $G[A]=K_{1,t}\cup I_{|A|-t-1}$ for some $t\ge 1$,
then $e(A_+) = |A_+| -1$.
Combining with (\ref{eq-6}), we get
$e(B)=0$ and $A_0 = \varnothing$.
Denote by $A=\{v_0,v_1,\ldots ,v_t\}$
the vertex set of the star $K_{1,t}$,
where $v_0$ is the center vertex.
If $t=1$, then $A=\{v_0,v_1\}$
forms an edge.
Recall that each vertex of $B$ has degree at least $2$, we have $N(w)=\{v_0,v_1\}$ for every $w\in B$.
Then $x_{v_0}=x_{v_1} > x_u$,
which contradicts with the maximality of $x_u$.
Next, we consider the case $t\ge 2$.
We claim that $d_B(v_0) =0$.
Otherwise, if $v_0w \in E(G)$ for some $w\in B$,
then $N(w) =\{v_0\}$
since $G$ is $F_2$-free.
This leads to a contradiction because each vertex of $B$ has
degree at least $2$.
Therefore, for every $w\in B$, we have $N(w) \subseteq
\{v_1,v_2, \ldots ,v_t\}$.
In other words, the vertex sets $A\setminus \{v_0\}$ and $B$ form
a (not necessarily complete) bipartite subgraph in $G$;
see the previous Figure \ref{fig-4}.
By Lemma \ref{lem31},
it follows that $\lambda < \frac{1+\sqrt{4m-3}}{2}$,
a contradiction.

\medskip
{\bf Subcase 2.2.}
If $G[A]=K_3\cup I_{|A|-3}$, then $e(A_+)= |A_+|=3$.
So we obtain from (\ref{eq-6}) that $0\le e(B)\le 1-\sum_{v\in A_0} {x_v}/{x_u}$.
Thus, we get two possibilities, namely,
\[ 
\text{either  $e(B)\in \{0,1\}$ if $A_0 = \varnothing$,
or $e(B)=0$ if $A_0 \neq \varnothing$.} \]

In the former case, that is, $A_0 = \varnothing$,
then $G[A]$ is a copy of $K_3$.
In fact, the graph $G$ can be determined uniquely. Indeed,
since $G$ is $F_2$-free, we notice that
each vertex of $B$ has at most one neighbor in $A$.
Recall that $G$ is connected and $d(w)\ge 2$ for every $w\in B$.
Therefore, $e(B)=1$ and $|B|=2$, so $m=9$ and
$G$ is described as below.
Denote by $A=\{v_1,v_2,v_3\}$
and $B=\{w_1,w_2\}$, where $\{u,v_1,v_2,v_3\}$
forms a copy of $K_4$ and $v_1w_1w_2v_3$
forms a path on $4$ vertices.
In this graph, we observe that $x_{v_1}=x_{v_3} > x_u$,
contradicting with the maximality of $x_u$.

In the latter case $A_0 \neq \varnothing$, we have $|A|\ge 4$
and $e(B)=0$.
Since $d(w)\ge 2$ for every $w\in B$, we get $m \ge 9$.
Moreover, each vertex of $B$ has at most one neighbor in $A_+$,
and has at least one neighbor in $A_0$.
In addition, inequality (\ref{eq-6}) reduces to
\begin{equation} \label{eq-7}
\sum_{v\in A_0} x_v \le x_u.
\end{equation}
Denote by $A_+= \{v_1,v_2,v_3\}$ the vertex set of the copy of $K_3$
in $G[A]$. We write $N(A_+)$ for the set of vertices which are adjacent to a vertex of $A_+$.
Note that
$ \lambda x_u = x_{v_1} + x_{v_2} +x_{v_3} + \sum_{v\in A_0} x_v $.
Using (\ref{eq-7}), we obtain
\begin{align}  \notag
\lambda^2 x_u &= \lambda (x_{v_1} +x_{v_2} + x_{v_3}) +
\lambda \sum_{v\in A_0}  x_v \\
&\le
3x_u + 2(x_{v_1} +x_{v_2} + x_{v_3}) + \sum_{w \in N(A_+)\cap B} x_w + \lambda x_u.  \label{eq-8}
\end{align}
If $N(A_+) \cap B = \varnothing$,
then $\lambda ^2 x_u <3 x_u + 6x_u + \lambda x_u$,
which yields $\lambda^2- \lambda < 9$.
Since $B \neq \varnothing$ and $e(B)=0$, we have $N_G(w)=N_A(w)=N_{A_0}(w)$ for every $w\in B$. 
Combining with $d(w)\ge 2$,
we get $|A_0|\ge 2$ and then $m\ge 10$.
Therefore, we obtain $\lambda^2- \lambda <m-1 $,
which contradicts with $\lambda \ge \frac{1+\sqrt{4m-3}}{2}$. 

Next, we assume that $N(A_+) \cap B \neq \varnothing$.
Since $G$ is $F_2$-free, we see that each vertex $w\in B$
has at most one neighbor in $A_+$.
Then $\sum_{w\in  N(A_+)\cap B} x_w < e(A_+,B)x_u \le |B|x_u$.
Recall that $e(B)=0$ and each vertex of $B$ has at least two neighbors in
$A$. Then
$m=|A| +3 + e(A,B) \ge |A| +3 +2|B|$,
which implies $|B| \le \lfloor \frac{m-|A|-3}{2} \rfloor$.
Combining with (\ref{eq-8}), we get
\begin{equation}  \label{eq-xiazheng}
\lambda^2 < 9 + |B| + \lambda \le
9 + \left\lfloor \frac{m-|A|-3}{2} \right\rfloor + \lambda.
\end{equation}
Recall  that $|A|\ge 4$.
If $m\ge 12$, then $\lambda^2 - \lambda < 9+
\lfloor \frac{m-7}{2} \rfloor \le m-1$.
This contradicts with the assumption $\lambda \ge \frac{1+\sqrt{4m-3}}{2}$.
Next, we consider the case $9\le m\le 11$.

\begin{figure}[H]
\centering
\includegraphics[scale=0.85]{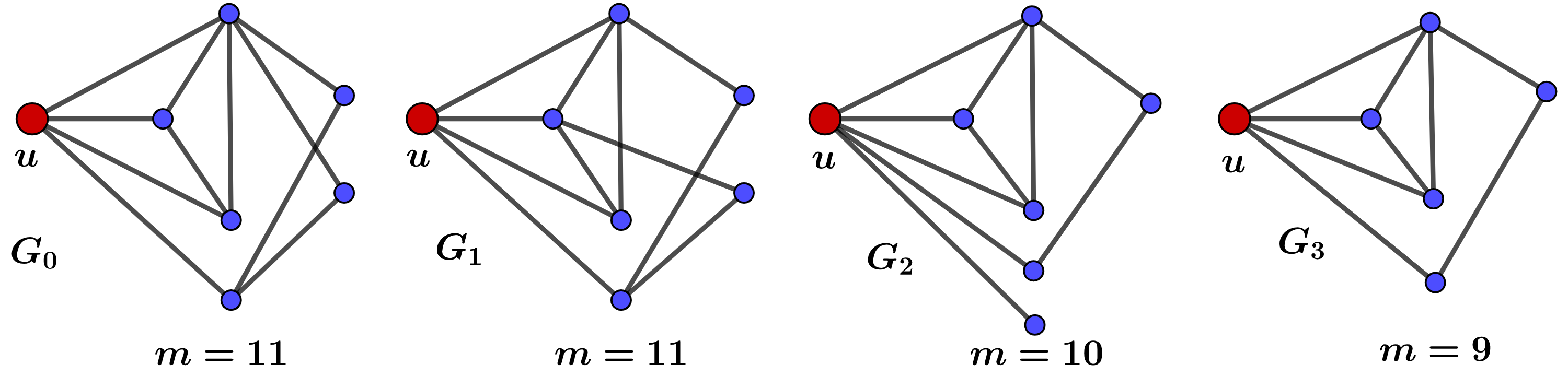}
\caption{The graphs $G_0,G_1,G_2$ and $G_3$.}
\label{fig-3}
\end{figure}

For the case $m=11$, observe that $|B| \le \lfloor \frac{m-7}{2}\rfloor =2$. If $|B|=1$, then it follows from (\ref{eq-xiazheng}) that
$\lambda^2 < 9 +1+ \lambda$,
and so $\lambda < \frac{1+\sqrt{41}}{2}$, a contradiction.
If $|B|=2$, then $G$ is represented as $G_0$ or $G_1$ in Figure  \ref{fig-3}.
By a direct calculation, we can verify that $\lambda (G_0) \approx 3.408< \frac{1+\sqrt{41}}{2} \approx 3.701$, 
and $\lambda (G_1) \approx
3.385 < \frac{1+\sqrt{41}}{2} \approx 3.701$,
a contradiction. 
For the case $m=10$ and $m=9$, the graph $G$ is determined as
$G_2$ and $G_3$, respectively in Figure \ref{fig-3}.
By calculations, we have $\lambda (G_2)\approx 3.315 < \frac{1+\sqrt{37}}{2} \approx 3.541$, a contradiction.
Moreover, we have $\lambda (G_3) \approx 3.236 < \frac{1+\sqrt{33}}{2} \approx 3.372$, a contradiction.
\end{proof}

\section{Concluding remarks}

\label{sec5}

In Conjecture \ref{conj-m-Fk},
we proposed a problem for $F_k$-free graphs with given size.
On the other hand, we would like to mention
the following conjecture involving consecutive cycles.
The extremal graphs in Conjectures  \ref{conj-m-Fk}
and \ref{conj-ZLS}  are surprisingly the same.

\begin{conjecture}[Zhai--Lin--Shu \cite{ZLS2021}, 2021]
\label{conj-ZLS}
Let $k$ be a fixed positive integer and $G$ be a graph of sufficiently large size $m$ without
isolated vertices. If $\lambda(G)\ge \frac{k-1+\sqrt{4m-k^2+1}}{2 }$, then $G$ contains a cycle of length $t$ for every $t \le 2k + 2$, unless $G= K_k\vee I_{\frac{1}{k}\left(m-{k \choose 2} \right)}$.
\end{conjecture}

In this section, we will conclude more  spectral extremal problems involving 
the friendship graph. 
Recall that Nosal's theorem \cite{Nosal1970} asserts that
every $m$-edge graph $G$ with $\lambda (G)> \sqrt{m}$
must contain a triangle. Moreover, a theorem of Nikiforov \cite{Niki2007laa2} implies that
every $n$-vertex graph $G$ with $\lambda (G) > \lambda (K_{\lceil \frac{n}{2} \rceil, \lfloor \frac{n}{2} \rfloor}) $ has a triangle.
In 2023, Ning and Zhai \cite{NZ2021} generalized these theorems and
proved two counting results of triangles,
which states that under the same assumption on spectral radius,
the graph $G$ contains not only one triangle, but also
at least $\lfloor \frac{\sqrt{m}-1}{2} \rfloor$
and $\lfloor \frac{n}{2}\rfloor -1$ triangles.
In this paper, we have determined
the largest spectral extremal graphs of the bowtie  
for graphs with given order and size, respectively. 
A natural question one may ask is that how many copies of
the bowtie must have in a graph with the spectral radius larger than
that of the extremal graph in Theorems \ref{thm-n-F2} or  \ref{thm-m-F2}.  
In fact, the edge supersaturation for the bowtie was studied by Kang, Makai and
Pikhurko \cite{KMP2020}. Roughly speaking, 
they proved that if $n$ is large enough and $q=o(n^2)$, 
then the only way to construct a graph on $n$ vertices with $\mathrm{ex}(n,F_2) +q$ edges 
and containing as few bowties as possible is to add $q-1$ edges to $K_{\lceil \frac{n}{2} \rceil, \lfloor \frac{n}{2} \rfloor}$. 
Hence it is also interesting to establish the spectral analogues 
for the bowtie. 
Let $K_{\lceil \frac{n}{2} \rceil, \lfloor \frac{n}{2} \rfloor}^{2K_2}$ 
be the graph obtained from $K_{\lceil \frac{n}{2} \rceil, \lfloor \frac{n}{2} \rfloor}$ by embedding two disjoint edges $2K_2$ 
into the part of size $\lceil \frac{n}{2} \rceil$. 

\begin{problem}
If $G$ is an $n$-vertex graph with 
$\lambda (G) \ge \lambda (K_{\lceil \frac{n}{2} \rceil, \lfloor \frac{n}{2} \rfloor}^{2K_2})$, 
then $G$ contains at least $\lfloor \frac{n}{2} \rfloor$ copies of bowties, 
and $K_{\lceil \frac{n}{2} \rceil, \lfloor \frac{n}{2} \rfloor}^{2K_2}$ is the unique spectral extremal graph. 
\end{problem}

\medskip
The definition of the spectral radius was recently extended
to the $p$-spectral radius; see \cite{KLM2014, KN2014,Niki2014laa} and references therein.
We denote the $p$-norm of $\bm{x}$ by
$ \lVert \bm{x}\rVert_p
=(\sum_{i=1}^n |x_i|^p )^{1/p}$.
For every real number $p\ge 1$,
the {\it $p$-spectral radius} of $G$ is defined as
\begin{equation*} \label{psp}
\lambda^{(p)} (G) : = 2 \max_{\lVert \bm{x}\rVert_p =1}
\sum_{\{i,j\} \in E(G)} x_ix_j.
\end{equation*}
We remark that $\lambda^{(p)}(G)$ is a versatile parameter.
Indeed, $\lambda^{(1)}(G)$ is known as the Lagrangian function of $G$,
$\lambda^{(2)}(G)$ is the spectral radius of its adjacency matrix,
and
\begin{equation}  \label{eqlimit}
\lim_{p\to +\infty} \lambda^{(p)} (G)=2e(G).
\end{equation}

The extremal  function for $p$-spectral radius is given as
\[  \mathrm{ex}_{\lambda}^{(p)}(n, {F}) :=
\max\{ \lambda^{(p)} (G) : |G|=n ~\text{and $G$ is ${F}$-free}  \}.  \]
To some extent, the proof of  results on the $p$-spectral radius shares some similarities with the usual spectral radius when $p>1$;
see \cite{KNY2015} for more extremal problems.
In 2014, Kang and Nikiforov \cite{KN2014} extended the Tur\'{a}n theorem
to the $p$-spectral version for $p>1$.
They proved that
$ \mathrm{ex}_{\lambda}^{(p)} (n,K_{r+1}) =\lambda^{(p)}(T_r(n))$,
and the unique extremal graph is the $r$-partite Tur\'{a}n graph $T_r(n)$.
This result unifies the classical Tur\'{a}n theorem and Nikiforov's theorem by the fact (\ref{eqlimit}).

Recall that $F_k$ is the graph consisting of $k$ triangles
by intersecting a common vertex.
Motivated by the above-mentioned results, we propose the following
extremal  problems in terms of the $p$-spectral radius for $F_k$-free graphs. 
In particular, we make the following problem on the case $k=2$.

\begin{problem}
If $p>1$ and $G$ is an $F_2$-free graph of large order $n$, then
\[  \lambda^{(p)} (G) \le \lambda^{(p)} (K_{ \lfloor \frac{n}{2} \rfloor, \lceil \frac{n}{2} \rceil}^+), \]
equality holds if and only if $G=K_{ \lfloor \frac{n}{2} \rfloor, \lceil \frac{n}{2} \rceil}^+$.
\end{problem} 

More generally, we can propose the following problem. 

\begin{problem}
Let $F$ be a graph with $\mathrm{ex}(n,F)=e(T_r(n)) + O(1)$. 
Let $p>1$ be a fixed real number and $n$ be a sufficiently large integer. 
If $G$ has the maximal $p$-spectral radius over all $n$-vertex $F$-free graphs, then 
$G\in \mathrm{Ex}(n,F)$. 
\end{problem}

There is a rich history on the study of bounding
the eigenvalues of a graph in terms of
various parameters and matrices.
Apart from the study of adjacency matrices,
it is also worth mentioning that the signless Laplacian matrices were  widely investigated over the years.
Recall that  $S_{n,k}=K_k \vee I_{n-k}$.
Clearly, we can see that
$S_{n,k}$  does not contain $F_k$ as a subgraph.
In 2021, Zhao, Huang and Guo \cite{ZHG21}
proved that $S_{n,k}$ is
the unique graph attaining the maximum signless Laplacian spectral radius among all $F_k$-free graphs of order $n\ge 3k^2-k-2$.
For  related extensions, we infer the readers to \cite{CLZ2021}
for intersecting odd cycles, \cite{WZ2022dm} for fan graphs.
For more other results, see, e.g., the $K_{r+1}$-free graphs \cite{HJZ2013},
the $C_{2k+1}$-free graphs \cite{Yuan2014} and
the $C_{2k}$-free graphs \cite{NY2015}.

\subsection*{Acknowledgements}
This work was supported by  NSFC (Grant No. 11931002 and 12001544),
and Natural Science Foundation of Hunan
Province (No. 2021JJ40707). 
The authors would like to thank Dr. Xiaocong He for carefully reading an early manuscript of this paper.

\end{document}